\documentclass[11pt,reqno]{amsart}

\usepackage{color}
\usepackage[colorlinks=true, allcolors=blue]{hyperref}
\usepackage{amsmath, amssymb, amsthm}
\usepackage{mathrsfs}
\usepackage{mathtools}
\usepackage[noabbrev,capitalize,nameinlink]{cleveref}
\crefname{equation}{}{}
\usepackage{fullpage}
\usepackage[noadjust]{cite}
\usepackage{graphics}
\usepackage{pifont}
\usepackage{tikz}
\usepackage{bbm}
\usepackage[T1]{fontenc}

\usetikzlibrary{arrows.meta}

\usepackage{environ}
\usepackage{framed}
\usepackage{url}
\usepackage[linesnumbered,ruled,vlined]{algorithm2e}
\usepackage[noend]{algpseudocode}
\usepackage[labelfont=bf]{caption}
\usepackage{cite}
\usepackage{framed}
\usepackage[framemethod=tikz]{mdframed}
\usepackage{appendix}
\usepackage{graphicx}
\usepackage[textsize=tiny]{todonotes}
\usepackage{tcolorbox}
\usepackage{enumerate}
\allowdisplaybreaks[1]
\usepackage{enumerate}

\DeclareSymbolFont{symbolsC}{U}{pxsyc}{m}{n}
\SetSymbolFont{symbolsC}{bold}{U}{pxsyc}{bx}{n}
\DeclareFontSubstitution{U}{pxsyc}{m}{n}
\DeclareMathSymbol{\medcircle}{\mathbin}{symbolsC}{7}

\crefname{algocf}{Algorithm}{Algorithms}

\crefname{equation}{}{} 
\AtBeginEnvironment{appendices}{\crefalias{section}{appendix}} 

\usepackage[color,final]{showkeys} 

\colorlet{refkey}{orange!20}
\colorlet{labelkey}{blue!30}

\crefname{algocf}{Algorithm}{Algorithms}

\numberwithin{equation}{section}
\newtheorem{theorem}{Theorem}[section]

\newtheorem{lemma}[theorem]{Lemma}

\crefname{claim}{Claim}{Claims}

\newtheorem*{question*}{Question}

\theoremstyle{definition}
\newtheorem{definition}[theorem]{Definition}

\newtheorem*{definition*}{Definition}

\theoremstyle{remark}
\newtheorem{remark}[theorem]{Remark}

\newcommand{\eps}{\varepsilon}

\newcommand{\mb}{\mathbb}

\let\originalleft\left
\let\originalright\right
\renewcommand{\left}{\mathopen{}\mathclose\bgroup\originalleft}
\renewcommand{\right}{\aftergroup\egroup\originalright}

\allowdisplaybreaks

\title{Enumerating Matroids and Linear Spaces}

\author[Kwan]{Matthew Kwan}
\address{Institute of Science and Technology Austria, 3400 Klosterneuburg, Austria}
\email{matthew.kwan@ist.ac.at}

\author[A2]{Ashwin Sah}
\author[A3]{Mehtaab Sawhney}
\address{Department of Mathematics, Massachusetts Institute of Technology, Cambridge, MA 02139, USA}
\email{\{asah,msawhney\}@mit.edu}

\begin{document}

\maketitle
\begin{abstract}
We show that the number of linear spaces on a set of $n$ points and the number of rank-3 matroids on a ground set of size $n$ are both of the form $(cn+o(n))^{n^2/6}$, where $c=e^{\sqrt 3/2-3}(1+\sqrt 3)/2$. This is the final piece of the puzzle for enumerating fixed-rank matroids at this level of accuracy: there are exact formulas for enumeration of rank-1 and rank-2 matroids, and it was recently proved by van der Hofstad, Pendavingh, and van der Pol that for constant $r\ge 4$ there are $(e^{1-r}n+o(n))^{n^{r-1}/r!}$ rank-$r$ matroids on a ground set of size $n$.
\end{abstract}

\section{Introduction}

\emph{Matroids} (also sometimes known as \emph{combinatorial geometries}) are fundamental objects that abstract the combinatorial properties of linear independence in vector spaces. Specifically, a matroid consists of a ground set $E$ and a collection $\mathcal{I}$ of subsets of $E$ called independent sets\footnote{Instead of defining a matroid by its collection of independent sets, some authors prefer to define a matroid by some other (equivalent) data, such as its collection of \emph{flats}, its collection of \emph{hyperplanes}, or its \emph{rank function}. See for example \cite{Wel76,Oxl11} for a more thorough introduction to matroids and their various definitions.}. The defining properties of a matroid are that:
\begin{itemize}
\item{the empty set is independent (that is, $\emptyset\in\mathcal{I}$);}
\item{subsets of independent sets are independent (if $A'\subseteq A\subseteq E$
and $A\in\mathcal{I}$, then $A'\in\mathcal{I}$);}
\item{if $A$ and $B$ are independent sets, and $|A|>|B|$,
then an independent set can be constructed by adding an element of
$A\setminus B$ to $B$ (there is $a\in A\backslash B$ such that $B\cup\{ a\} \in\mathcal{I}$).}
\end{itemize}

Observe that any finite set of elements in a vector space (over any
field) naturally gives rise to a matroid, though most matroids do not
arise this way. The \emph{rank} of a matroid is the maximum size of an independent set.

Enumeration of matroids is a classical topic, though the state of our knowledge is rather incomplete. Some early upper and lower bounds on the total number of matroids on a ground set of size $n$ were obtained in the 1970s by Piff and Welsh~\cite{PW71}, Piff~\cite{Pif73} and Knuth~\cite{Knu74}, and these bounds were improved only recently by Bansal, Pendavingh, and van der Pol~\cite{BPP15}. It is also of interest to enumerate matroids of fixed rank: let $m(n,r)$ be the number of rank-$r$ matroids on a ground set of size $n$. It is trivial to see that $m(n,1)=2^n-1$, and it is also possible to prove the exact identity $m(n,2)=b(n + 1)-2^n$, where $b(m)$ is the $m$th Bell number (which counts the number of partitions of an $m$-element set). This identity seems to have been first observed by Acketa~\cite{Ack78}.

For $r\ge 3$, an exact expression for $m(n,r)$ in terms of well-known functions does not seem to be possible\footnote{Though, a lot of computational work has been done for small $n,r$, and there are many conjectures about the relations between the different $m(n,r)$; see for example \cite{Duk04} and the index for matroids on the On-Line Encyclopedia of Integer Sequences~\cite{oeis}.}, but after some exciting recent developments, rather precise asymptotic expressions have become available.
First, Pendavingh and van der Pol~\cite{PP17} observed that (for constant $r\ge 1$) the lower bound $m(n,r)\ge(e^{1-r}n+o(n))^{n^{r-1}/r!}$ follows from Keevash's breakthrough work~\cite{Kee14,Kee18} on existence and enumeration of \emph{combinatorial designs}.
They also proved an upper bound of the form $m(n,r)\le(en+ o(n))^{n^{r-1}/r!}$. Even more recently, van der Hofstad, Pendavingh and van der Pol~\cite{HPP18} closed the gap for all $r\ne 3$, proving that $m(n,r)=(e^{1-r}n+o(n))^{n^{r-1}/r!}$ for constant $r\ge 4$. In the remaining case $r=3$ they were able to prove $m(n,3)\le(ne^{1+\beta}+o(n))^{n^{2}/6}\approx (1.4n)^{n^2/6}$, where $-0.67<\beta<-0.65$ is the solution to a certain variational problem. In this paper, we close the gap completely in this case $r=3$.

\begin{theorem}\label{thm:matroids}
\[m(n,3)=\Big(\frac{1+\sqrt 3}{2}e^{\sqrt 3/2-3}n+o(n)\Big)^{n^2/6}\approx (0.16169n)^{n^2/6}.\]
\end{theorem}
In particular, \cref{thm:matroids} shows that the lower bound $m(n,3)\ge(e^{-2}n+o(n))^{n^2/6}$ obtainable from Keevash's results is far from sharp. This confirms a conjecture in \cite{HPP18} (and disproves the earlier \cite[Conjecture 8.2.9]{Pol17}).

In fact, \cref{thm:matroids} is really a corollary of the following theorem, estimating the number of \emph{linear spaces} on a set of $n$ points. In incidence geometry, a linear space on a point set $P$ is a collection of subsets of at least two points of $P$ (called \emph{lines}) such that each pair of points lies in a unique line (see for example \cite{BB93,Shu11} for more on linear spaces). For reasons that will become clear in a moment, we denote the number of linear spaces on a set of $n$ points by $p(n,3)$.
\begin{theorem}\label{thm:partitions}
\[p(n,3)=\Big(\frac{1+\sqrt 3}{2}e^{\sqrt 3/2-3}n+o(n)\Big)^{n^2/6}\approx (0.16169n)^{n^2/6}.\]
\end{theorem}

We remark that one may also be interested in linear spaces in which no line has exactly 2 points (these are called \emph{proper} linear spaces). It should be possible to adapt our proof to show that the expression in \cref{thm:partitions} is also a valid estimate for the number of proper linear spaces on a set of $n$ points (though this would require some rather deep machinery due to Keevash~\cite{Kee14} and McKay and Wormald~\cite{MW90}). See \cref{rem:keevash} for discussion.

To explain the connection between \cref{thm:partitions,thm:matroids} we need to make a few more definitions. A \emph{$d$-partition} (or \emph{generalised partition of type $d$}) of a ground set $E$ is a collection of subsets of $E$ (called \emph{parts}) each having size at least $d$, such that every subset of $d$ elements of $E$ is contained in exactly one of the parts. So, a 1-partition is an ordinary partition, and a 2-partition is a linear space. For any $r\ge 2$, there is a correspondence between the set of $(r-1)$-partitions of $E$ and the set of so-called \emph{paving matroids} of rank $r$ on the ground set $E$. Namely, a paving matroid of rank $r$ is a matroid for which its set of \emph{hyperplanes} (maximal subsets with rank $r-1$) form an $(r-1)$-partition of its ground set. See for example \cite[Section~3]{Wel76} for more details.

For $r\ge 2$ let $p(n,r)$ be the number of paving matroids of rank $r$, or equivalently the number of $(r-1)$-partitions, on a ground set of size $n$. Given the above correspondence, we trivially have $p(n,r)\le m(n,r)$, and it was proved by Pendavingh and van der Pol~\cite[Theorem~3]{PP17} that $p(n,r)\le m(n,r)\le p(n,r)^{1+O(1/n)}$ for constant $r$. So \cref{thm:matroids} is a direct consequence of \cref{thm:partitions}, and for the rest of the paper we will abandon the language of matroids and focus on \cref{thm:partitions}.

In fact, we find it convenient to use the language of graph theory: note that a linear space on a set of $n$ points is precisely equivalent to a \emph{clique-decomposition} of the complete graph $K_n$ (meaning, a decomposition of the edges of $K_n$ into nonempty cliques of arbitrary sizes).

\subsection{Discussion of proof techniques}
If one is interested in counting the number of decompositions of $K_n$ into cliques which each have a \emph{fixed} number of vertices $k$, this is a problem about enumerating \emph{combinatorial designs}. Specifically, such a decomposition corresponds exactly to a design called a \emph{$(2,k,n)$-Steiner system}. Such Steiner systems can be enumerated using powerful tools due to Keevash~\cite{Kee14,Kee18}: in particular, if $n$ satisfies certain necessary divisibility conditions, the number of such Steiner systems can be written as $(c_kn+o(n))^{\alpha_k n^2}$, where $c_k^{k-2}=e^{1-\binom k2}/(k-2)!$ and $\alpha_k=(k-2)/(k(k-1))$. Note that $\alpha_k$ is maximised for two different $k$: namely, when $k=3$ and when $k=4$. This suggests that decompositions containing mostly 3-cliques and 4-cliques comprise the bulk of the clique-decompositions in $p(n,3)$.

The above observation motivates our proof strategy, and we believe it also explains why counting $(r-1)$-partitions and rank-$r$ matroids is most difficult when $r=3$ (if $r\ge 4$, then one can do a similar calculation for \emph{hypergraph} clique-decompositions and see that there is a single maximising value of $k$).

For the lower bound in \cref{thm:partitions} (namely, that there are \emph{at least} about $(0.16 n)^{n^2/6}$ clique-decompositions), we proceed in a very similar fashion as in \cite{Kee18}: we consider a random process that builds a clique-decomposition by iteratively removing random 3-cliques and 4-cliques from $K_n$ (with a particular carefully chosen ratio between the two), until a very small number of edges remain (these edges are then treated as 2-cliques in our decomposition). We then study the number of possible outcomes of this process\footnote{It would be possible to consider a random process that, at each step, randomly decides whether to remove a 3-clique or 4-clique, with an appropriate probability. This would be in close correspondence with the upper bound approach described below. However, it is more convenient for us to reuse existing analysis of clique removal processes, and consider the concatenation of a 4-clique removal process and a 3-clique removal process.}. The details of the lower bound appear in \cref{sec:lower}.

The upper bound is more interesting. In \cite{Kee18}, Keevash is able to upper-bound the number of Steiner systems by adapting an approach of Linial and Luria~\cite{LL13}, using the so-called \emph{entropy method}. Roughly speaking, the idea is as follows. To prove an upper bound on the number of $k$-clique decompositions of $K_n$, it suffices to prove an upper bound on the entropy of a uniformly random $k$-clique decomposition $P$. In order to specify an outcome of $P$, it suffices to specify, for each edge $e\in K_n$, the clique $C_e$ containing $e$. Therefore, one can upper-bound the entropy of $P$ by considering an ordering $e_1,\dots,e_{\binom n2}$ of the edges of $K_n$, and upper-bounding the conditional entropy of each $C_{e_i}$, given the previous cliques $C_{e_1},\dots,C_{e_{i-1}}$. If $e_1,\dots,e_{\binom n 2}$ is a \emph{random} ordering, then it is possible to upper-bound these conditional entropies by studying the expected number of possible choices for $C_{e_i}$ given $C_{e_1},\dots,C_{e_{i-1}}$. This is possible due to a certain symmetry of $k$-clique decompositions: namely, Keevash makes crucial use of the fact that in \emph{any} $k$-clique decomposition, for any edge $e$ and almost all $k$-cliques $C\subseteq K_n$ containing $e$ there are exactly $(\binom k2-1)^2$ edges $e'\notin C$ such that $C_{e'}$ and $C$ share an edge (meaning that after $C_{e'}$ is revealed, $C$ can be ruled out as a possible outcome of $C_e$).

For decompositions of $K_n$ into cliques of mixed sizes, an analogous symmetry property does not hold, and the number of edges $e'$ whose clique $C_{e'}$ intersects a particular clique $C$ depends on the structure of our clique-partition. So, we cannot prove the upper bound in \cref{thm:partitions} by a straightforward generalisation of Keevash's proof. Instead, we exploit a different symmetry property of clique-decompositions, generalising an observation in \cite{Kwa20}, as follows. Suppose $P$ is a clique-partition into cliques of bounded size (say, each of the cliques in $P$ has at most 100 vertices). Then, if we take the union of a \emph{random} subset of the cliques in $P$, where each clique is included independently with probability $p\in (0,1)$, we are very likely to arrive at a \emph{quasirandom} graph with density about $p$ (i.e., a graph whose ``local statistics'' resemble a random subgraph of $K_n$ obtained by including each edge with probability $p$ independently). Sweeping some details under the rug, this means that we can give an upper bound on the number of ways to choose a clique-decomposition with a prescribed number of cliques of each size (the precise statement is in \cref{lem:prescribed-count}), by counting in a clique-by-clique manner, where at each step the number of choices for a $k$-clique is roughly the expected number of $k$-cliques in a random graph of the appropriate density\footnote{In our actual proof, due to the fact that we are considering decompositions into cliques of mixed sizes, it is more convenient to consider small ``chunks'' of cliques with a representative number of cliques of each size, and estimate the number of choices for each chunk, rather than estimating the number of choices for each clique individually.}. We remark that our approach seems to be more flexible than the entropy method, for problems of this type: it is possible to recover all of Keevash's upper bounds in this way (though with slightly weaker quantitative aspects). Also, in our view, our clique-by-clique approach is more naturally in correspondence with the clique-by-clique processes used to prove lower bounds in this area.

Finally, having an upper bound for the number of clique partitions with a prescribed number $s_k$ of $k$-cliques for each $k\le 100$ (and no cliques with more than 100 vertices), it remains to show that the contribution from cliques with more than 100 vertices is negligible, and to optimise our formula over choices of $s_1,\dots,s_{100}$. For the former, we use a very crude encoding argument (\cref{lem:reduction}). The latter is a simple calculus exercise (essentially, we use the method of Lagrange multipliers; see \cref{lem:optimisation}). In agreement with the heuristic mentioned earlier, we find that our formula is maximised when only $s_3,s_4$ are non-negligible.

\subsection{Notation}
We use standard asymptotic notation throughout, as follows. For functions
$f=f(n)$ and $g=g(n)$, we write $f=O(g)$
to mean that there is a constant $C$ such that $|f|\le C|g|$,
$f=\Omega(g)$ to mean that there is a constant $c>0$
such that $f(n)\ge c|g(n)|$ for sufficiently large $n$, and $f=o(g)$ to mean that $f/g\to0$ as $n\to\infty$.
Also, following \cite{Kee14}, the notation $f=1\pm\varepsilon$ means
$1-\varepsilon\le f\le1+\varepsilon$.

We write $N_G(v)$ to denote the neighbourhood of a vertex $v$ in a graph $G$ (i.e., the set of vertices adjacent to $v$). For a real number $x$, the floor and ceiling functions are denoted $\lfloor x\rfloor=\max\{i\in \mb Z:i\le x\}$ and $\lceil x\rceil =\min\{i\in\mb Z:i\ge x\}$. We will however mostly omit floor and ceiling symbols and assume large numbers are integers, wherever divisibility considerations are not important. Finally, all logarithms in this paper are in base $e$.

\subsection*{Acknowledgements}
We thank Michael Simkin for helpful comments on the manuscript. We thank Zach Hunter for several corrections. Sah and Sawhney were supported by NSF Graduate Research Fellowship Program DGE-1745302. Sah was supported by the PD Soros Fellowship.

\section{The upper bound}\label{sec:upper}
\subsection{Removing the contribution from large parts}
We first reduce to the case where all cliques have bounded size. Related ideas appeared in \cite{HPP18}.

\begin{lemma}\label{lem:reduction}
Fix $L\ge 11$ and $n$ sufficiently large as a function of $L$. Let $\Gamma_{s_2,\ldots,s_L;E}$ denote the set of clique-decompositions of $K_n$, for which there are $E$ edges covered by cliques with more than $L$ vertices, and there are $s_k$ cliques with $k$ vertices for each $2\le k\le L$. Then
\[|\Gamma_{s_2,\ldots,s_L;E}|\le n^{|E|/5}|\Gamma_{s_2 + E,s_3,\ldots,s_L;0}|.\]
\end{lemma}
\begin{proof}
Fix a clique decomposition $P\in \Gamma_{s_2,\ldots,s_L,E}$. Let $P_1$ be the ``truncated'' clique decomposition obtained from $P$ by first removing each clique with more than $L$ vertices, and then adding two-vertex cliques (i.e., single edges) for each of the edges which are no longer covered by a clique. Then $P$ is uniquely determined by the pair $(P_1,P_2)$, where $P_2$ contains all the cliques in $P$ with more than $L$ vertices.

There are at most $2^{|E|-1}$ ways to choose a sequence $s_{L+1},\dots,s_n$ such that $\sum_{t=L+1}^{n} \binom t 2 s_t=|E|$. Indeed, such a sequence can be interpreted as an integer partition\footnote{The number of partitions of an integer $N$ is at most its number of compositions, which is $2^{N-1}$.} of $|E|$ (where we are only allowed to use parts which have size of the form $\binom{t}{2}$ for $t > L$). For each such $s_{L+1},\dots,s_n$, the number of possibilities for $P_2$ which contain exactly $s_t$ cliques of each size $t> L$ is at most $\prod_{t=L+1}^{n}(n^t)^{s_t}=n^\Sigma$, where \[\Sigma=\sum_{t=L+1}^{n} t s_t\le \frac 2{L}\sum_{t=L+1}^{n} \binom t 2 s_t=\frac{2|E|}{L}.\]
Then, we observe that $2^{|E|-1}n^{\Sigma}\le n^{|E|/5}$ for $L\ge 11$ and $n$ large enough.
\end{proof}

\subsection{Counting decompositions into prescribed numbers of bounded-size cliques}
We now estimate the number of clique-decompositions with prescribed numbers of cliques of each (bounded) size. We will later optimise over choices of these prescribed numbers.
\begin{lemma}\label{lem:prescribed-count}
Fix a constant $L\in \mb N$ and integers $s_2,\dots,s_L$ such that $\binom 2 2s_2+\dots+\binom L 2 s_L=\binom n 2$. Let $\Gamma_{s_2,\dots,s_L}$ be the set of all clique-decompositions of $K_n$ whose number of $t$-cliques is $s_t$ for each $t$. Then
\[\Gamma_{s_2,\dots,s_L}\le \exp\bigg(\sum_{k=2}^L s_k\bigg(\log \binom n k-\log s_k+1\bigg)-\binom n 2\pm n^{2-\Omega(1)}\bigg).\]
\end{lemma}
We define an \emph{ordered} clique-decomposition of $K_n$ to be an ordered list of cliques whose edge-disjoint union is equal to $K_n$. Let $\Xi_{s_2,\dots,s_L}$ be the set of all orderings of clique-decompositions in $\Gamma_{s_2,\dots,s_L}$. First, we need the following modification of \cite[Lemma~2.6]{Kwa20}, showing that for initial segments of a random ordered clique-decomposition, the graph of uncovered edges is ``typical''/``quasirandom''.

\begin{lemma}\label{lem:random-order}
Fix a constant $L\in \mb N$ and any integers $s_2,\dots,s_L$ such that $\binom 2 2s_2+\dots+\binom L 2 s_L=\binom n 2$. Consider a uniformly random ordered clique-decomposition (of $K_n$) from $\Xi_{s_2,\dots,s_L}$ (which has $N:=s_2+\dots+s_L$ cliques). Let $G_m$ be the random subgraph of $K_n$ consisting of the edges not appearing in the first $m$ cliques of our random ordered clique-decomposition. Then with probability $1-o_{n\to\infty}(1)$, for all $0\le m\le N$ and all sets of vertices $A$ with $|A|\le L$, we have
\[\bigg|\bigg|\bigcap_{w\in A}N_{G_m}(w)\bigg|-(1-m/N)^{|A|}n\bigg|\le n^{1/2}\log n.\]
\end{lemma}
\begin{proof}
Fix a particular choice of $m$ and $A$; we will take a union bound over all such choices. It suffices to consider a uniformly random ordering of a \emph{fixed} clique-decomposition $P\in \Gamma_{s_2,\dots,s_L}$ (i.e., we prove the desired statement conditioned on any outcome of the unordered set of cliques in our random ordered clique-decomposition).

The first $m$ cliques in our random ordering comprise a uniformly random subset $R\subseteq P$ of $m$ cliques in $P$. Consider the closely related ``binomial'' random subset $R'\subseteq P$, where each clique is included with probability $1-m/N$ independently; let $G_m'$ contain the edges of $K_n$ which do not appear in any of the cliques in $R'$.

Since the cliques in $P$ are edge-disjoint, note that there are at most $\binom{|A|}{2}=O(1)$ cliques in $P$ that include more than one vertex in $A$. Let $U$ be the set of vertices in these atypical cliques. Now, for each $v\notin A$ and $w\in A$ there is exactly one clique $e_{v}^{w}$ in $P$ containing $v$ and $w$, whose presence in $R'$ would prevent $v$ from appearing in $\bigcap_{v\in A}N_{G_m'}(v)$. For each fixed $v\notin U$ the hyperedges $e_{v}^{w}$, for $w\in A$, are distinct, so
\[\Pr\bigg(v\in\bigcap_{w\in A}N_{G_m'}(w)\bigg)=(1-m/N)^{|A|}.\]
Let $Q=|\bigcap_{w\in A}N_{G_m'}(w)|$; it follows that $\mb E Q=(1-m/N)^{|A|}n+O(1)$.

Now let $I$ be the set of cliques of $P$ which contain a vertex of $A$. We have $|I| = O(n)$ since the cliques in $P$ are edge-disjoint. Note that $Q$ is entirely determined by $I\cap R'$, and adding or removing any clique from $R'$ affects $Q$ by at most $L-1=O(1)$. So by the Azuma--Hoeffding inequality, for $t\ge 0$ we have
\[\Pr(|Q-\mb E Q|\ge t)\le 2\exp(-\Omega(t^2/n)).\]
It follows that with probability at least $1-n^{-10L}$ we have $|Q-(1-m/N)^{|A|}n|\le \sqrt n \log n$. Recall that we have been considering the ``binomial'' random subset $R'$; we can transfer this result to the ``uniform'' random subset $R$ using a standard inequality (for example, the so-called Pittel inequality; see \cite[p.~17]{JLR00}). Then, we take a union bound over choices of $m$ and $A$.
\end{proof}
We also need the fact that the cliques of different sizes are ``well-distributed'' in a random ordered clique-decomposition.
\begin{lemma}\label{lem:random-order-2}
Fix a constant $L\in \mb N$ and integers $s_2,\dots,s_L$ such that $\binom 2 2s_2+\dots+\binom L 2 s_L=\binom n 2$. Consider a uniformly random ordered clique-decomposition (of $K_n$) from $\Xi_{s_2,\dots,s_L}$ (which has $N:=s_2+\dots+s_L$ cliques). Then with probability $1-o_{n\to\infty}(1)$, for any $0\le m<m'\le N$ and any $0\le k\le L$, if we consider all the cliques ranging from the $(m+1)$-th to the $m'$-th in our random ordered clique-decomposition, the number of such cliques that have exactly $k$ vertices differs from $s_k(m'-m)/N$ by at most $n\log n$.
\end{lemma}
\begin{proof}
As in the proof of \cref{lem:random-order}, it suffices to consider a uniformly random reordering of a \emph{fixed} clique-decomposition $P\in \Xi_{s_2,\dots,s_L}$. The desired result then follows from a concentration inequality for the hypergeometric distribution (see for example \cite[(2.5) and Theorem~2.10]{JLR00}) and the union bound.
\end{proof}

Now we are ready to prove \cref{lem:prescribed-count}.

\begin{proof}[Proof of \cref{lem:prescribed-count}]
Let $N=s_2+\dots+s_L$, and let $c$ be a very small constant ($c=1/(10L^2)$ will do). We will count \emph{ordered} clique decompositions in $\Xi_{s_2,\dots,s_L}$, and then at the end of the proof we will divide by $N!$.

Partition the interval $\{1,\dots,N\}$ into sub-intervals $I_1,\dots,I_{n^c}$ by taking \[I_i=\{1,\dots,N\}\cap ((i-1)Nn^{-c},iNn^{-c}].\]
Let $m_i=\min I_i=\lfloor(i-1)Nn^{-c}+1\rfloor$ be the first index in each $I_i$. Say that an ordered clique decomposition $P\in \Xi_{s_2,\dots,s_L}$ is \emph{ordinary} if for each $1\le i\le n^c$, the following hold.
\begin{enumerate}
    \item The graph $G^{(i)}:=G_{m_i-1}$ consisting of those edges not covered by the first $m_i-1$ cliques of $P$ satisfies the conclusion of \cref{lem:random-order}.
    \item For each $1\le i\le n^c$ and $2\le k\le L$, the number of cliques ranging from the $m_i$-th to the $(m_{i+1}-1)$-th which have exactly $k$ vertices satisfies the conclusion of \cref{lem:random-order-2}.
\end{enumerate}
Almost all ordered clique decompositions in $\Xi_{s_2,\dots,s_L}$ are ordinary by \cref{lem:random-order,lem:random-order-2}, so it suffices to prove an upper bound on the number of ordinary decompositions.

For each $1\le i\le n^c$, we consider separately the number of choices for the cliques indexed by indices in $I_i$, for an ordinary ordered clique-decomposition. Let $\gamma_i=1-(i-1)n^{-c}$. Now, (1) implies that for all $k\le L$, the number of $k$-cliques in $G^{(i)}$ is
\begin{equation}\gamma_i^{\binom k 2}n^k/k!+O(n^{k-1/2}\log n).\label{eq:clique-count}\end{equation}
To see this, we count the number of ways to choose an ordered list of $k$ vertices inducing a clique, in a vertex-by-vertex fashion, then divide by $k!$.

For any choice of $t_k=n^{-c}s_k+O(n\log n)$, we have
\[\binom{t_2+\cdots+t_L}{t_2,\ldots,t_L} = e^{O(n(\log n)^2)}\bigg(\prod_{k=2}^L\bigg(\frac{s_k}{s_2+\cdots+s_k}\bigg)^{-\frac{s_k}{s_2+\cdots+s_k}}\bigg)^{n^{-c}(s_2+\cdots+s_k)+O(n\log n)}.\]
So, given (2), we can multiply these estimates to see that the number of ways to choose the cliques indexed by $I_i$ is at most
\[e^{n^{2-c-\Omega(1)}} \prod_{k=2}^L\Big(\frac{s_k}{N}\Big)^{-s_k n^{-c}}\prod_{k=2}^L\Big(\gamma_i^{\binom k 2} \frac{n^k}{k!}\Big)^{n^{-c}s_k}.\]

We next take the product of this expression over all $1\le i\le n^c$, and divide by the number of orderings $N!$ of each clique-decomposition, to obtain the desired result. In particular one obtains 
\begin{align*}
&e^{n^{2-\Omega(1)}}\frac{1}{N!}\prod_{i=1}^{n^c} \bigg(\prod_{k=2}^L\Big(\frac{s_k}{N}\Big)^{-s_k n^{-c}}\prod_{t=2}^L\Big(\gamma_i^{\binom k 2} \frac{n^k}{k!}\Big)^{n^{-c}s_k}\bigg) \\
&= e^{n^{2-\Omega(1)}}\frac{1}{N!}\prod_{k=2}^L\Big(\Big(\frac{s_k}{N}\Big)^{-s_k}\Big(\frac{n^k}{k!}\Big)^{s_k}\Big)\cdot\prod_{i=1}^{n^c} \prod_{t=2}^L\gamma_i ^{n^{-c}\binom{k}{2}s_k}\\
&= e^{n^{2-\Omega(1)}-\binom{n}{2}}\prod_{k=2}^L\Big(\frac{s_k}{e}\Big)^{-s_k}\Big(\frac{n^k}{k!}\Big)^{s_k}.
\end{align*}
We note that this involves an approximation by a Riemann integral:
\begin{align*}
\prod_{i=1}^{n^c}\prod_{k=2}^L\gamma_i^{n^{-c}s_k} = \exp\bigg(n^{-c}\sum_{i=1}^{n^c}\log \gamma_i \sum_{k=2}^L\binom{k}{2}s_k \bigg)&=\exp\bigg(\binom{n}{2}n^{-c}\sum_{i=1}^{n^c} \log(in^{-c}) \bigg)\\
&= \exp\bigg(\binom{n}{2}\bigg(\int_0^1 \log x~dx\pm O(n^{-c/2})\bigg)\bigg)\\
&=\exp\bigg(-\binom{n}{2} + O(n^{2-c/2})\bigg).\qedhere
\end{align*}
\end{proof}

\subsection{Optimising over prescribed clique numbers}
Given \cref{lem:reduction,lem:prescribed-count}, the upper bound in \cref{thm:partitions} will be a simple consequence of the following lemma.
\begin{lemma}\label{lem:optimisation}
Fix a constant $L\in \mb N$, let $D$ be the set of $(s_2,\dots,s_L)\in \mb R^{L-1}$ such that $s_2,\dots,s_L\ge 0$ and $\binom 22s_2+\dots+\binom L2s_L=\binom n2$, and consider the real-valued function $f:D\to\mb R$ defined by
\[f(s_2,\dots,s_L)=\frac{1}{5}s_2\log n+\sum_{k=2}^L\bigg(s_k\log\binom{n}{k}-s_k\log s_k+s_k\bigg)-\binom n 2.\]
Then, the maximum value of $f(s_2,\dots,s_L)$ is
\[\frac{n^2}{6}\bigg(\log n -3 + \frac{\sqrt{3}}{2} + \log\frac{1+\sqrt{3}}{2} \pm n^{-\Omega(1)}\bigg).\]
\end{lemma}
\begin{proof}
Since $f$ is continuous on $D$ and $D$ is compact, our function $f$ attains a maximum.

First, we claim that a maximum can only be attained when all $s_k$ are strictly positive. Indeed, consider some $(s_1,\dots,s_L)\in D$ for which $s_k=0$, in which case there must be some $s_j>0$. We make a slight perturbation: increase $s_k$ to $\binom{j}{2}\eps$ and decrease $s_j$ by $\binom{k}{2}\eps$, for some very small $\eps > 0$ (note that we are still in $D$), and consider the corresponding change to the value of $f$. Note that the terms containing $s_j$ decrease by $O(\eps)$ but the terms containing $s_k$ increase by $\Omega(\eps\log(1/\eps))$. So, our perturbation has increased the value of $f$, which proves the claim.

Note that if we increase $s_k$ by $\binom{j}{2}\eps$ and decrease $s_j$ by $-\binom{k}{2}\eps$, for some very small $\varepsilon>0$, then the value of $f$ increases by
\[\binom{j}{2}\bigg(\log\binom{n}{k}-\log s_k+\tau(k)\bigg)\eps-\binom{k}{2}\bigg(\log\binom{n}{j}-\log s_j+\tau(j)\bigg)\eps + O(\eps^2),\]
where we set $\tau(k)=(\log n)/5$ when $k=2$, and $\tau(k)=0$ when $k\ne 2$. (This essentially follows from taking a derivative).
So, a maximum can only occur when each
\[\frac{\log\binom{n}{k}-\log s_k+\tau(k)}{\binom{k}{2}}\]
takes a common value $\lambda$. For this $\lambda$, we see that
\begin{equation}s_k = \binom{n}{k}e^{-\binom{k}{2}\lambda+\tau(k)}.\label{eq:sk}\end{equation}
Now, recalling the definition of $D$, we have
\begin{equation}n^{1/5}\binom{n}{2}e^{-\lambda}+\sum_{k=3}^L\binom{k}{2}\binom{n}{k}e^{-\binom{k}{2}\lambda} = \binom{n}{2}.\label{eq:lagrange}\end{equation}
There is a unique $\lambda$ satisfying this equation, because the left-hand side of the equation is monotonically decreasing in $\lambda$. Now, if $\lambda = (\log n)/3 + \alpha$, for $|\alpha|\le 1$, then we may compute
\[n^{1/5}\binom{n}{2}e^{-\lambda}+\sum_{k=3}^L\binom{k}{2}\binom{n}{k}e^{-\binom{k}{2}\lambda}=(e^{-3\alpha}/2+e^{-6\alpha}/4)n^2 \pm n^{2-\Omega(1)}.\]
(The dominant terms of the expression on the left-hand side are the ones with $k\in\{3,4\}$). Therefore, we can write the solution to \cref{eq:lagrange} as $\lambda=(\log n)/3+\alpha$ for $|\alpha|\le 1$ (since $\alpha = -1$ makes the left side of \cref{eq:lagrange} too large and $\alpha = 1$ makes it too small).

Recall that $\binom n 2=(1/2)n^2+O(n)$, so for the $\lambda$ satisfying \cref{eq:lagrange} we have
\[e^{-3\alpha}/2+e^{-6\alpha}/4 = 1/2 + n^{-\Omega(1)},\]
hence the quadratic formula yields
\[e^{-3\alpha} = \sqrt{3}-1+n^{-\Omega(1)}.\]
Thus
\[e^{-3\lambda}=(\sqrt{3}-1)n^{-1}\pm n^{-1-\Omega(1)}.\]
Using \cref{eq:sk}, we then compute that $f$ is maximised when
\begin{align*}
    s_3 &= \frac{n^3}{6}e^{-3\lambda}\pm n^{2-\Omega(1)} = \frac{n^2(\sqrt{3}-1)}{6}\pm n^{2-\Omega(1)},\\
    s_4 &= \frac{n^4}{24}e^{-6\lambda}\pm n^{2-\Omega(1)} = \frac{n^2(2-\sqrt{3})}{12}\pm n^{2-\Omega(1)},
\end{align*}
and $s_k = n^{2-\Omega(1)}$ for $k\notin\{3,4\}$. The desired result follows after substituting into the formula for $f(s_2,\dots,s_L)$ and simplifying.
\end{proof}

\subsection{Deducing the upper bound}
We now give the short deduction of the upper bound in \cref{thm:partitions} using \cref{lem:reduction,lem:prescribed-count,lem:optimisation}.

\begin{proof}[Proof of the upper bound in \cref{thm:partitions}]
Let $L=11$. The sets $\Gamma_{s_2,\dots,s_L;E}$ defined in \cref{lem:reduction} form a partition of the set of all clique-decompositions of $K_n$. There are at most ${\binom n2}^{L}=e^{n^{2-\Omega(1)}}$ choices of $s_2,\dots,s_L,E$, so it suffices to upper-bound the maximum possible value of $|\Gamma_{s_2,\dots,s_L,E}|$. By \cref{lem:reduction} it in fact suffices to upper-bound the maximum possible value of $n^{s_2/5}|\Gamma_{s_2,\dots,s_L,0}|$. This is precisely what is accomplished by \cref{lem:prescribed-count,lem:optimisation}.
\end{proof}

\section{The lower bound}\label{sec:lower}
The lower bound in \cref{thm:partitions} is an immediate consequence of the following estimate.
\begin{lemma}\label{lem:lower-bound}
Let
\[s_3= \Big\lfloor\frac{n^2(\sqrt{3}-1)}{6}\Big\rfloor,\qquad s_4 = \Big\lfloor\frac{n^2(2-\sqrt{3})}{12}\Big\rfloor.\]
For $c>0$, let $\Gamma_c$ be the collection of clique-decompositions of $K_n$ in which there are $s_3-\lfloor n^{2-c}\rfloor$ cliques with 3 vertices, $s_4$ cliques with 4 vertices, and the rest are cliques with 2 vertices. If $c>0$ is sufficiently small then
\[|\Gamma_c|\ge\Big(\frac{1+\sqrt 3}{2}e^{\sqrt 3/2-3}n-o(n)\Big)^{n^2/6}.\]
\end{lemma}

To prove \cref{lem:lower-bound} we need a notion of ``typicality'' (called ``quasirandomness'' in \cite{Kwa20}), closely related to the property in \cref{lem:random-order}.
\begin{definition}\label{def:typical}
For an $n$-vertex, $m$-edge graph $G$, we define its \emph{density} $p(G)=m/\binom n2$. We say that $G$ is \emph{$(\varepsilon,h)$-typical} if for every set $A$ of at most $h$ vertices of $G$, the vertices in $A$ have $(1\pm \varepsilon)p(G)^{|A|}n$ common neighbours.
\end{definition}
Note that if an $n$-vertex graph $G$ with density $p$ is $(\varepsilon,h)$-typical then it has
\[(1\pm O_k(\varepsilon))p^{\binom k 2}n^k/k!\]
$k$-cliques, for any $k\le h$ (as for \cref{eq:clique-count}, we count cliques vertex-by-vertex; this calculation also appears explicitly in \cite[Proposition~2.8]{Kwa20}).

\begin{proof}
Given a graph $G$ and an integer $k$, we define its \emph{$K_k$-removal process} as follows. Starting from the graph $G$, at each step we consider the set of all copies of $K_k$ in our graph, choose one uniformly at random, and remove its edges. (Eventually we will run out of copies of $K_k$, at which point the process aborts).

We will need the following facts about the behaviour of the $K_3$-removal process and the $K_4$-removal process.
\begin{enumerate}
    \item There is $a>0$ such that the following holds. If we run the $K_4$-removal process on $K_n$, then with probability $1-o(1)$:
    \begin{enumerate}
        \item the process does not abort before $s_4$ steps, and
        \item for each $t\le s_4$, the graph at step $t$ is $(n^{-a},3)$-typical.
    \end{enumerate}
    \item For every $a>0$ there is $c>0$ such that the following holds. Let $G$ be an $n$-vertex graph with $m:=\binom n2-6s_4=(3-O(1/n))s_3$ edges which is $(n^{-a},2)$-typical. If we run the $K_3$-removal process on $G$, then with probability $1-o(1)$:
    \begin{enumerate}
        \item the process does not abort before $s_3-n^{2-c}$ steps, and
        \item for each $t\le s_3-n^{2-c}$, the graph at step $t$ is $(n^{-c},2)$-typical.
    \end{enumerate}
\end{enumerate}

For a simple proof of Fact (2), see \cite[Theorem~4.1]{Kwa20}. Fact (1) can be proved in basically exactly the same way (in fact it is slightly simpler, because we start from the complete graph instead of a general typical graph). See \cite[Section~6]{Kee18} for some discussion of (a generalisation of) the $K_k$-removal process starting from a complete graph, which implies the desired result.

Now, we simply concatenate the $K_4$-removal process and the $K_3$-removal process. Indeed, starting from the complete graph $K_n$, we first run $s_4$ steps of the $K_4$-removal process, then $s_3-n^{2-c}$ steps of the $K_3$ removal process. In this way, either we abort or we produce a clique decomposition in $\Gamma_c$, in which our set of 4-cliques and our set of 3-cliques are both equipped with an ordering. Let $Q$ be the set of outcomes of our concatenated process for which in each of the first $s_4$ steps, our graph is $(n^{-b},3)$-typical, and in each of the next $s_3-n^{2-c}$ steps, our graph is $(n^{-c},2)$-typical.

The probability of each outcome in $Q$ is at most
\begin{equation}\prod_{t=1}^{s_4} \bigg((1-n^{-\Omega(1)})\bigg(\frac{\binom n 2-6t}{\binom n2}\bigg)^6\frac{n^4}{24}\bigg)^{-1}\prod_{t=1}^{s_3-n^{2-c}} \bigg((1-n^{-\Omega(1)})\bigg(\frac{\binom n 2-6s_4-3t}{\binom n2}\bigg)^3\frac{n^3}{6}\bigg)^{-1},\label{eq:count-outcomes}\end{equation}
and by (1--2) above, these probabilities sum up to $1-o(1)$. So, the number of outcomes in $Q$ is at least $1-o(1)$ divided by the expression in \cref{eq:count-outcomes}. It follows that
\[|\Gamma_c|\ge \Big(\frac{n^4}{24}\Big)^{s_4}\Big(\frac{n^3}{6}\Big)^{s_3}\exp\bigg(6\sum_{t=1}^{s_4} \log\bigg(\frac{\binom n 2-6t}{\binom n2}\bigg)+3\sum_{t=1}^{s_3} \log \bigg(\frac{\binom n 2-6s_4-3t}{\binom n2}\bigg)-n^{2-\Omega(1)}\bigg)/(s_4!s_3!).\]
(The difference between taking a sum up to $s_3$ and up to $s_3-n^{2-c}$ is easily seen to contribute to the negligible $\exp(-n^{2-\Omega(1)})$ factor.)

Let $a_3 = 3s_3/\binom{n}{2}$ and $a_4 = 6s_4/\binom{n}{2}$, and note that $a_3+a_4 = 1-O(1/n)$. By Stirling's approximation we compute that $\log |\Gamma_c|$ is at least
\begin{align*}
&\sum_{k=3}^{4}s_k\bigg(\log\binom{n}{k}-\log s_k + 1\bigg)+6\sum_{t=1}^{s_4} \log\bigg(\frac{\binom n 2-6t}{\binom n2}\bigg)+3\sum_{t=1}^{s_3} \log \bigg(\frac{\binom n 2-6s_4-3t}{\binom n2}\bigg) - n^{2-\Omega(1)}\\
&\qquad= \sum_{k=3}^{4}s_k\bigg(\log\binom{n}{k}-\log s_k + 1\bigg)\\
&\qquad\qquad\qquad+6\binom{n}{2}\int_{0}^{a_4/6}\log(1-6t)~dt+3\binom{n}{2}\int_{0}^{a_3/3} \log (1-a_4-3t)~dt - n^{2-\Omega(1)}\\
&\qquad= \sum_{k=3}^{4}s_k\bigg(\log\binom{n}{k}-\log s_k + 1\bigg)+\binom{n}{2}\int_{0}^{a_3+a_4}\log(1-t)~dt- n^{2-\Omega(1)}\\
&\qquad= \sum_{k=3}^{4}s_k\bigg(\log\binom{n}{k}-\log s_k + 1\bigg)-\binom{n}{2}- n^{2-\Omega(1)}.
\end{align*}
Substituting the values of $s_3$ and $s_4$ and simplifying (or alternatively, comparing with the expressions in the proof of \cref{lem:optimisation}) yields the desired result.
\end{proof}

\begin{remark}\label{rem:keevash}
In \cref{lem:lower-bound} we consider clique-decompositions that have a small number of ``trivial'' cliques with two vertices. We believe that it is possible to adapt the proof to avoid such cliques, but this requires  some of Keevash's deepest results on clique-decompositions of quasirandom graphs. Namely, for any constant $k$, Keevash's machinery~\cite{Kee18} allows one to estimate the number of $K_k$-decompositions of any dense quasirandom graph satisfying certain divisibility conditions (the number of edges should be divisible by $\binom k2$ and every degree should be divisible by $k-1$; say such a graph is \emph{$K_k$-divisible}). So, in order to prove a version of \cref{lem:lower-bound} in which no clique has exactly two vertices (thereby proving a version of \cref{thm:partitions} for proper linear spaces), it suffices to prove a suitable lower bound on the number of ways to partition the edges of $K_n$ into a $K_3$-divisible quasirandom graph with density $(\sqrt 3-1)/2+o(1)$, a $K_4$-divisible quasirandom graph with density $(2-\sqrt 3)/2+o(1)$ and a tiny ``remainder graph'' with $O(1)$ edges, itself decomposable into cliques which have more than two vertices. A suitable lower bound on the number of such graph partitions can be proved with some elementary number theory and the machinery of McKay and Wormald~\cite{MW90} for enumerating graphs with a given dense degree sequence (the remainder graph is just to handle divisibility issues, and it turns out we can always choose it to be either a copy of $K_5$, a copy of $K_7$, or a vertex-disjoint union $K_5\cup K_7$).
\end{remark}

\bibliographystyle{amsplain0.bst}
\bibliography{main.bib}

\end{document}